\documentclass[11pt,letterpaper]{amsart}
\usepackage{amsmath}
\usepackage{amssymb}
\usepackage{amsthm}
\usepackage{hyperref}
\usepackage[noabbrev,capitalize]{cleveref}
\usepackage{mathrsfs}
\usepackage{mathtools}
\usepackage{slashed}
\usepackage{tikz-cd}
\usepackage[shortlabels]{enumitem}

\setenumerate{label=(\roman*)}

\DeclareMathOperator{\Ric}{Ric}

\def\sideremark#1{\ifvmode\leavevmode\fi\vadjust{\vbox to0pt{\vss
 \hbox to 0pt{\hskip\hsize\hskip1em
 \vbox{\hsize3cm\tiny\raggedright\pretolerance10000
 \noindent #1\hfill}\hss}\vbox to8pt{\vfil}\vss}}}

\newtheorem{theorem}{Theorem}[section]
\newtheorem{proposition}[theorem]{Proposition}
\newtheorem{lemma}[theorem]{Lemma}
\newtheorem{corollary}[theorem]{Corollary}

\theoremstyle{definition}
\newtheorem{definition}[theorem]{Definition}

\theoremstyle{remark}
\newtheorem{remark}[theorem]{Remark}

\numberwithin{equation}{section}

\makeatletter
\@namedef{subjclassname@2020}{\textup{2020} Mathematics Subject Classification}
\makeatother

\begin{document}

\title[Nonexistence of the metric with positive $Cm$ and $H_m$]{Nonexistence of the metric with positive intermediate curvatures on manifolds with boundary}

\author{Jingche Chen}
\address{Department of Mathematical Sciences, Tsinghua University, 100084, Beijing, China}
\email{cjc23@mails.tsinghua.edu.cn}

\author{Han Hong}
\address{Department of Mathematics and statistics \\ Beijing Jiaotong University \\ Beijing \\ China, 100044}
\email{hanhong@bjtu.edu.cn}

\begin{abstract}
We establish curvature obstruction theorems for manifolds with boundary. Our main theorems show that, for dimensions up to 7, a topologically nontrivial compact manifold with boundary cannot have a metric of positive 
$m$-intermediate curvature if the boundary is $m$-convex, and some rigidity result holds if $m$-intermediate curvature is nonnegative. This non-existence persists after performing a connected sum with an arbitrary manifold. These results generalize results of \cite{brendlegeroch'sconjecture,chenshuli_end,ChuKwongLeerigidity,XuTrans} to manifold with boundary.
\end{abstract}

\maketitle

\section{Introduction}

The classical Geroch conjecture asks whether the torus $\mathbb{T}^{n}$ admits a Riemannian metric with positive scalar curvature. This conjecture was resolved by Schoen and Yau \cite{Schoen-Yau-PSC} for $3\leq n\leq 7$ using the minimial hypersurface method and by Gromov and Lawson \cite{Gromov-Lawson-psc} using the spin method for all dimensions. Recently, Brendle, Hirsch and Johne \cite{brendlegeroch'sconjecture} extended this result to $m$-intermediate curvature. Their approach utilizes stable $m$-weighted slicings to demonstrate that for $n\leq 7$ and $1\leq m\leq n-1$, the manifold $M^{n-m}\times \mathbb{T}^{m}$ does not admit a metric of positive $m$-intermediate curvature. Subsequently, Chu, Kwong and Lee \cite{ChuKwongLeerigidity} studied rigidity results when the $m$-intermediate curvature
is non-negative and the ambient dimension is at most $5$. Xu \cite{XuTrans} also obtained rigidity results, showing that they  hold for $n=6$ and constructing  counterexamples for $n\geq 7$. 

There is also significant interest in extending the Geroch conjecture to the non-compact settings, such as whether positive scalar curvature exists on $\mathbb{T}^{n}\# X$, where $X$ is an arbitrary manifold. By analyzing complete non-compact stable minimal hypersurfaces, Lesourd, Unger and Yau \cite{Lesourd-Unger-Yau} proved the case $n = 3$.  Chodosh and Li \cite{chodoshlisoapbubble} later applied the $\mu$-bubble technique to address this problem for $3\leq n\leq 7$. Their method also applies to $(\mathbb{S}^{1}\times M^{n-1})\# X$ for $3 \leq n \leq 7$, where $M$ is a Schoen-Yau-Schick manifold and $X$ is arbitrary. In \cite{chenshuli_end}, Chen refined their method and showed that for $3\leq n\leq 7$, $(\mathbb{T}^{m}\times M^{n-m})\# X$, where $M$ is a closed manifold and $X$ is an arbitrary manifold, does not admit a metric of positive $m$-intermediate curvature for  $m,n$ satisfying certain algebraic inequalities. For general dimensions, Wang and Zhang \cite{Wang-Zhang-geroch} showed that for any $n$ and any spin $n$-manifold $X$, the connected sum $\mathbb{T}^{n}\#X$ admits no complete metric of positive scalar curvature. 

The paper aims to explore analogous results for compact manifolds with boundary. 
Existing results for manifolds with boundary include work by Gromov and Lawson \cite{Gromov-Lawson-ann}, who showed that if $\mathbb{T}^{n}\times [0,1]$ admits a metric $g$ with positive scalar curvature, then the mean curvature $H_{g}$ on the boundary must be negative somewhere. R\"ade \cite{Rade} and J. Xu \cite{xu2025scalarmeancurvaturecomparison} have generalized Gromov-Lawson's result, proving various rigidity theorems.

Our first result extends the non-existence results by Brendle, Hirsch and Johne \cite{brendlegeroch'sconjecture} and the  rigidity results by Chu, Kwong, Lee \cite{ChuKwongLeerigidity} and Xu \cite{XuTrans}.
\begin{theorem}\label{main theorem 1 Introduction}
    Let $(M,\partial M)$ be a compact orientable $n$-dimensional manifold. Suppose that there exist nonzero $  \alpha^{1},...,\alpha^{m}\in H^{1}(M,\mathbb{Z})$ such that 
    $(\alpha^{1}\smile\cdots\smile \alpha^{m})|_{\partial M}\neq 0$. Then

    (1) If $2\leq n\leq 7$ and $1\leq m\leq n-1$, there is no Riemmanian metric $g$ on $M$ satisfying $C_{m}> 0$ and $H_{m}\geq 0$, or $C_{m}\geq 0$ and $H_{m}> 0$.  

    (2) If $2\leq n\leq 6$ and $1\leq m\leq n-1$, or $n=7$ and $m\in \{1,5,6\}$, and there exists a metric $g$ on $M$ satisfying $C_{m}\geq 0$ and $H_{m}\geq 0$, then $(M,\partial M)$ is isometrically covered by $(X^{n-m}\times \mathbb{R}^{m},\partial X^{n-m}\times \mathbb{R}^{m})$, where $X$ is a compact manifold with boundary satisfying $\text{Ric}_{X}\geq 0$ and $A^{\partial X}\geq 0$.
\end{theorem}

The cohomology assumption $(\alpha^{1}\smile\cdots\smile \alpha^{m})|_{\partial M}\neq 0$ can imply $\alpha^{1}\smile\cdots\smile \alpha^{m}\neq 0$. If we weaken it to $\alpha^{1}\smile\cdots\smile \alpha^{m}\neq 0$, then for $2\leq n\leq 7$ and $1\leq m\leq n-1$, there is no Riemmanian metric $g$ on $M$ satisfying $C_{m}> 0$ and $H_{m}\geq 0$. In fact, under this assumption, it is not guaranteed that minimizers have boundaries. Thus only positive curvature in the interior plays role in the proof.

The part $(1)$ of Theorem \ref{main theorem 1 Introduction} recovers 
\cite[Theorem 1.5]{Chow_2023} by Chow, Johne and Wan  and \cite[Corollary 1.9]{wu2025comparisontheoremsintermediatericci} by Wu.

For the special case that $m=n-1$, we have $C_{m}=\frac{1}{2}R_{g}$ and $H_{m}=H_{g}$, yielding the following corollary.

\begin{corollary}\label{corollary of thm1 Introduction}
    Let $(M,\partial M)$ be a compact orientable $n$-dimensional manifold. Suppose that there exist nonzero $  \alpha^{1},...,\alpha^{n-1}\in H^{1}(M,\mathbb{Z})$ such that 
    $(\alpha^{1}\smile\cdots\smile \alpha^{n-1})|_{\partial M}\neq 0$. Then

    (1) If $2\leq n\leq 7$, there is no Riemmanian metric $g$ on $M$ satisfying $R_{g}> 0$ and $H_{g}\geq 0$, or $R_{g}\geq 0$ and $H_{g}> 0$.  

    (2) If $2\leq n\leq 7$ and there exists a metric $g$ on $M$ satisfying $R_{g}\geq 0$ and $H_{g}\geq 0$, then $(M,\partial M)$ is isometrically covered by $([0,1]\times \mathbb{R}^{n-1}, \{0,1\}\times \mathbb{R}^{n-1})$.
\end{corollary}

For $n=3$, the topological assumption implies existence of a free boundary stable minimal surface $\Sigma$ with nontrivial free boundary and Euler characteristic $\chi(\Sigma)\leq 0$. In this case, the result follows from  Ambrozio \cite{lucasrigiditythree}. The similar nonexistence results of Ambrozio in higher-dimension are studied by Barbosa, Conrado \cite{Barbosa-Conrado-disk} and Brendle, Hung \cite{brendle2024systolicinequalitieshorowitzmyersconjecture}. Their results imply part(1) of Corollary \ref{corollary of thm1 Introduction}. Our results also relate to R\"ade's work, which shows that $M$ is isometric to $[0,1]\times X$ under some topological conditions.

The second result in the following extends Chen's non-compact results from \cite{chenshuli_end}.

\begin{theorem}\label{main theorem 2 Introduction}
    Assume either $2\leq n\leq 5$, $1\leq m\leq n-1$, or $6\leq n\leq 7$, $m\in\{1,n-2,n-1\}$. Let $(M,\partial M)$ be a compact orientable $n$-dimensional manifold. Suppose there exist nonzero $  \alpha^{1},...,\alpha^{m}\in H^{1}(M,\mathbb{Z})$ such that $\alpha^{1}\smile\cdots\smile \alpha^{m}\neq 0$. Then for any $n$-dimensional manifold $N$, the connected sum $M\# N$ does not admit a complete Riemannian metric $g$ with positive $m$-intermediate curvature and nonnegative $m$-intermediate boundary curvature.
\end{theorem}

Taking $m=n-1$, we have $C_{m}=\frac{1}{2}R_{g}$ and $H_{m}=H_{g}$, leading to the following corollary.

\begin{corollary}\label{main corollary 2 Introduction}
    For $2\leq n\leq 7$, let $(M,\partial M)$ be an $n$-dimensional compact orientable manifold as in Theorem \ref{main theorem 2 Introduction}. Then for any $n$-dimensional manifold $N$, the connected sum $M\# N$ does not admit a complete Riemannian metric $g$ with positive scalar curvature and nonnegative boundary mean curvature.
\end{corollary}

Notice that in Theorem \ref{main theorem 2 Introduction} and Corollary \ref{main corollary 2 Introduction}, the connected sum $\#$ can denote either an interior or a boundary connected sum (see Definition \ref{interior connected sum} and Definition \ref{boundary connected sum}). Consequently, Corollary \ref{main corollary 2 Introduction} implies that for any point $p$ (whether in the interior or on the boundary), the punctured manifold $(\mathbb{T}^{n-1}\times [0,1])\setminus{p}$ does not admit a complete metric with positive scalar curvature and a mean convex boundary.

\vskip.2cm
The rest of paper is organized as following. In section \ref{section2}, we give some preliminaries about existence of weighted slicing and review $m$-intermediate curvature and $m$-intermediate boundary curvature. In section \ref{section 3} and section \ref{section 4}, we present the proofs of Theorem \ref{main theorem 1 Introduction} and Theorem \ref{main theorem 2 Introduction}, respectively.

\subsection{Acknowledgements}
 The first author would like to thank his supervisor Prof Haizhong Li for his support. The authors appreciate Dr. Shuli Chen and Prof. Jianchun Chu for reading the first draft and providing valuable comments that improve the paper. The authors would also thank Prof. Brendle and Prof. Ambrozio for explanations about their papers. The second author is supported by NSFC No. 12401058 and the Talent
Fund of Beijing Jiaotong University No. 2024XKRC008. 

\section{Preliminary}\label{section2}

Firstly, we give some topological preliminaries. 
\begin{lemma}\label{represented}
    Let $M^{n}$ be a compact connected orientable smooth manifold with boundary. If there exists $0\neq \alpha\in H_{n-1}(M,\partial M;\mathbb{Z})$, then $\alpha$ is represented by a compact embedded orientable hypersurface $\Sigma$. Furthermore, if $\partial \Sigma$ is non-empty, $\partial\Sigma$ intersects $\partial M$ transversally.
\end{lemma}
\begin{proof}
    Notice that $S^{1}=K(\mathbb{Z}, 1)$, so we have $H^{1}(M;\mathbb{Z}) = [M, S^1]$, where $[M, S^1]$ are
    homotopy classes of maps from $M$ to $S^1$. Thus, we can choose a non-constant smooth
    map $f: M \rightarrow S^1$ representing the Lefschetz duality of $\alpha$ in $H^1(M;\mathbb{Z})$. Using Sard’s theorem, there exists $a$ as the regular value of $f$. Then $f^{-1}(a)$ is a compact embedded orientable hypersurface by the regular value theorem. By perturbing slightly, we can assume that the intersection of $\partial \Sigma$ and $\partial M$ is transversal.
\end{proof}
The next lemma describes separating hypersurface via relative homology.

\begin{lemma}\label{separate}
    Let $M^{n}$ be a compact connected orientable manifold with boundary and let $\Sigma^{n-1}\subset M$ be a compact embedded connected orientable hypersurface (possibly with boundary $\partial\Sigma\subset \partial M$).
    Then $\Sigma$ is separating if and only if $[\Sigma] = 0\in H_{n-1}(M,\partial M;\mathbb{Z})$.
\end{lemma}

\begin{proof}
    Firstly, we show $[\Sigma]=0$ implies that $\Sigma$ is separating. Suppose $\Sigma$ is non-separating, in the other word, $M\backslash \Sigma$ is connected. Thus, there exists a loop $S$ transversally intersect $\Sigma$ at exactly one point. Then the intersection number $\langle [S],[\Sigma]\rangle \neq 0$ which is a contradiction with $[\Sigma]=0$.
    
    Next, we show the reverse direction. Let $V=M\backslash \Sigma$, then $V=M_{+}\cup M_{-}$, where $M_{+},M_{-}$ are disjoint connected open sets. Next, we consider their boundary, we have $\tilde{\Sigma}_{+}\cup \tilde{\Sigma}_{-}\cup(\cup_{i=1}^{k} \Sigma_{i})=\Sigma_{+}\cup \Sigma_{-}\cup \partial M = \partial M_{+}\cup \partial M_{-}$, where $\Sigma_{\pm}$ is $\Sigma$ with orientation, $\tilde{\Sigma}_{\pm}$ are the connected components containing $\Sigma_{\pm}$, respectively and $\Sigma_{i}$ are some components of $\partial M$. Since $\tilde{\Sigma}_{+}\subset \partial M_{+}$ and $\tilde{\Sigma}_{-}\subset \partial M_{-}$, then $\partial M_{+} = \tilde{\Sigma}_{+}\cup (\cup_{j=1}^{k_{0}}\Sigma_{i_{j}})$. Thus, $[\Sigma]=0\in H_{n-1}(M,\partial M;\mathbb{Z})$.
\end{proof}

We review the the first variation and second variation of the free boundary minimizing hypersurface $\Sigma\hookrightarrow M$ with weighted area functional $ \mathcal{H}_{\rho}(\Sigma)=\int_{\Sigma}\rho\ d\mu$. Suppose $\Sigma_{s}$ is a proper variation of $\Sigma$ satisfying that $\partial \Sigma_{s}\subset \partial M$. Then
\[
\left.\frac{d}{ds}\right|_{s=0} \mathcal{H}_{\rho}(\Sigma_{s}) = \int_{\Sigma}\langle \rho H_{\Sigma}+\nabla_{M}\rho ,X^{\perp}\rangle\ d\mu+\int_{\partial \Sigma}\rho\langle X,\eta_\Sigma\rangle \ dA
\]
where $X$ is the variational vector field and $\eta_\Sigma$ is the outward conormal vector of $\partial \Sigma \subset \Sigma$.
Thus, for a critical point of $ \mathcal{H}_\rho$, we have that $H_{\Sigma}=-\langle \nabla_{M}\rho,X^{\perp}\rangle$ in $\Sigma$ and $\langle X,\eta_\Sigma \rangle = 0$ along $\partial \Sigma$. Let $X=f\nu$ where $\nu$ is the unit normal vector field of $\Sigma$ in $M$. The minimizer of $ \mathcal{H}_\rho$ is a stable free boundary minimal hypersurface. Then
\[
\begin{aligned}
\left.\frac{d^2}{ds^2}\right|_{s=0} \mathcal{H}_{\rho}(\Sigma_{s}) 
&= \int_{\Sigma}[-f\Delta_{\Sigma}f-(|A^{\Sigma}|^2+\text{Ric}_{M}(\nu,\nu))f^2-f\langle \nabla_{\Sigma}\log \rho,\nabla_{\Sigma}f\rangle\\
&+(\nabla^{2}_{M}\log \rho)(\nu,\nu)f^2]\rho \ d\mu+\int_{\partial \Sigma}f(\frac{\partial f}{\partial \eta_\Sigma}-fA^{\partial M}(\nu,\nu)) dA\geq 0.
\end{aligned}
\]
We can consider the positive first eigenfunction $f\in C^{2} (\Sigma)$ which satisfies the following equation.
\[
\left\{
\begin{array}{cc}
     -\Delta_{\Sigma}f-(|A^{\Sigma}|^2+\text{Ric}_{\log \rho}(\nu,\nu))f-\langle \nabla_{\Sigma}\log \rho,\nabla_{\Sigma}f\rangle = \lambda f, & \text{in}\ \Sigma \\
     \frac{\partial f}{\partial \eta_\Sigma} = fA^{\partial M}(\nu,\nu),& \text{on}\ \partial \Sigma 
\end{array}\right.
\]
where $\text{Ric}_{\log \rho}(\nu,\nu)=\text{Ric}_{M}(\nu,\nu)-(\nabla^{2}_{M}\log \rho)(\nu,\nu)$.
Notice that $\lambda\geq 0$.
Now we introduce the stable weighted slicing with free boundary of order $m$. In closed case, it refers to \cite{brendlegeroch'sconjecture}.
\begin{definition}
Suppose $1\leq m\leq n-1$ and let $(M,\partial M)$ be an orientable compact Riemannian manifold of dimension $n$. A stable free boundary weighted $m$-slicing consists of a collection of orientable and smooth submanifolds $\Sigma_{k}$, $0\leq k \leq m$, and a collection of positive functions $\rho_{k}\in C^{\infty}(\Sigma_{k})$ satisfying the following conditions:

(1) $\Sigma_{0}=M$ and $\rho_{0}=1$.

(2) For each $1\leq k\leq m$, $\Sigma_{k}$ is an $(n-k)$-dimensional, two-sided, embedded hypersurface in $\Sigma_{k-1}$ with boundary $\partial\Sigma_k$ intersecting $\partial\Sigma_{k-1}$ orthogonally. Moreover, $\Sigma_{k}$ is a stable critical point of the $\rho_{k-1}$ weighted area
\[
    \mathcal{H}_{\rho_{k-1}}^{n-k}(\Sigma) = \int_{\Sigma}\rho_{k-1}\ d\mu
\]
in the class of hypersurfaces $(\Sigma,\partial\Sigma)\subset (\Sigma_{k-1},\partial\Sigma_{k-1})$.

(3) For each $1\leq k\leq m$, the function $\frac{\rho_{k}}{\rho_{k-1}|_{\Sigma_{k}}}$ is a first eigenfunction of the stability operator associated with $\rho_{k-1}$ weighted area.
\end{definition}

The following lemma provides a sufficient condition for the existence of a stable weighted $k$-slicing with free boundary.

\begin{lemma}\label{existence of slicing}
    Suppose $M$ is an $n$-dimensional compact, orientable manifold with nonempty boundary and let $1\leq k\leq n-1$ where $n\leq 7$. Suppose $\alpha^{1},...,\alpha ^{k}\in H^{1}(M,\mathbb{Z})$ which are nontrivial satisfying  $\smile_{j=1}^{k}\alpha^{j}\neq 0$ and $\smile_{j=1}^{k}\alpha^{j}|_{\partial M}\neq 0$. Then there exists a stable free boundary weighted k-slicing $ \ \Sigma_{k}\subset \cdots\subset \Sigma_{1}\subset\Sigma_{0}=M$ such that for any $1\leq j\leq k$, we have $[\Sigma_{j}]=(\alpha^{j}\smile\cdots\smile \alpha^{1} )\frown [M]\in H_{n-j}(M,\partial M)$ and $[\partial \Sigma_{j}]=(-1)^{j}(\alpha^{j}\smile \cdots\smile \alpha^{1} )|_{\partial M}\frown [\partial M]\in H_{n-j-1}(\partial M)$. In particular, $\forall\ 1\leq j\leq k$, $\partial \Sigma_{j}\neq \emptyset$. 
\end{lemma}


\begin{proof}
    The first step is to construct $\Sigma_{1}\subset M$. Since $\alpha^{1}\neq 0$, then $\alpha^{1}\frown [M]\neq 0$ in $H_{n-1}(M,\partial M)$. Using the geometry measure theory, there exists $\Sigma_{1}\subset M$ which is a smooth free boundary minimizing hypersurface representing $\alpha^{1}\frown [M]$. Moreover, $[\partial \Sigma_{1}] =\partial (\alpha^{1}\frown [M]) = (\alpha^{1}|_{\partial M})\frown [\partial M]\neq 0\in H_{n-1}(\partial M) $.

    Next, we construct the k-slicing. Suppose we have constructed $\Sigma_{j}\subset \cdots\subset \Sigma_{1}\subset M$, we aim to construct $\Sigma_{j+1}\subset \Sigma_j$. Using similar argument in the construction of $\Sigma_{1}$, we obtain $H_{j+1}$ which is a hypersurface in $M$ representing $\alpha^{j+1}\frown [M]$. We use $P$ to denote the Lefschetz dual map. Then 
    \[
    \begin{aligned}
    &[H_{j+1}\cap \Sigma_{j}] = P(P^{-1}([H_{j+1}])\smile P^{-1}([\Sigma_{j}]))\\
    &= (\alpha^{j+1}\smile\cdots\smile\alpha^{1})\frown [M]\neq 0\in H_{n-j-1}(M,\partial M).
    \end{aligned}
    \]
    Then
    \[
    [H_{j+1}\cap \Sigma_{j}]\neq 0 \in H_{n-j-1}(\Sigma_{j},\partial \Sigma_{j}).
    \]
    Thus, there exists $\Sigma_{j+1}\subset (\Sigma_{j},\rho_{j}^{\frac{2}{n-j-1}}g)$ which is a free boundary minimizing hypersurface representing $[H_{j+1}\cap \Sigma_{j}]\in H_{n-j-1}(\Sigma_{j},\partial \Sigma_{j})$. In the other word, $\Sigma_{j+1}$ is a free boundary weighted minimizing hypersurface w.r.t. $\int_{\Sigma}\rho_{j}d\mu$. Moreover,
    \[
    [\Sigma_{j+1}]= [H_{j+1}\cap \Sigma_{j}]\in H_{n-j-1}(M,\partial M).
    \]
    and
    \[
    \begin{aligned}
    &[\partial \Sigma_{j+1}]= \partial [H_{j+1}\cap \Sigma_{j}]= \partial ((\alpha^{j+1}\smile\cdots\smile\alpha^{1})\frown [M])\\
    &=(-1)^{j+2}(\alpha^{j+1}\smile\cdots\smile\alpha^{1})|_{\partial M}\frown [\partial M])\in H_{n-j-2}(\partial M).
    \end{aligned}
    \]
\end{proof}
\begin{remark}\label{without boundary condition}
    If we do not need $\partial \Sigma_{j}\neq \emptyset$, the condition $\smile_{j=1}^{k}\alpha^{j}|_{\partial M}\neq 0$ can be discarded.
\end{remark}
\begin{remark}
    Since the boundary map is well defined, $\smile_{j=1}^{k}\alpha^{j}|_{\partial M}\neq 0$ can imply $\smile_{j=1}^{k}\alpha^{j}\neq 0$.
\end{remark}
\begin{remark}
    We give an example that satisfies the condition of Lemma \ref{existence of slicing}. Consider $[0,1]\times \mathbb{T}^{n}$ with the standard product metric and the coordinates $(x,\theta_{1},...,\theta_{n})$. Then there exist one-forms $d \theta_{i}$ satisfying $d\theta_{1}\wedge\cdots\wedge d\theta_{n} \neq 0$ and $d\theta_{1}\wedge\cdots\wedge d\theta_{n}|_{\{0,1\}\times \mathbb{T}^{n}}\neq 0$. The corresponding slicing with free boundary is $\{\theta_{1}=\theta_{1}(\alpha_{1}),...,\theta_{n}=\theta_{n}(\alpha_{n})\}\subset\{\theta_{2}=\theta_{2}(\alpha_{1}),...,\theta_{n}=\theta_{n}(\alpha_{n})\}\subset\cdots\subset \{\theta_{n}=\theta_{n}(\alpha_{n})\}\subset [0,1]\times\mathbb{T}^{n}$. Here $(\theta_{1}(\alpha_{1}),...,\theta_{n}(\alpha_{n}))$ is a fixed point belonging to $\mathbb{T}^{n}$.
\end{remark} 

In the proof of the main theorem, we shall need a well-known fact about the free boundary hypersurface that we state here.
\begin{lemma}\label{boundary}
    Let $(\Sigma,\partial \Sigma)\hookrightarrow (M,\partial M)$ be a free boundary hypersurface. Then
    
    (1) $\eta=\eta_{\Sigma}$.
    
    (2) $A^{\partial M}(X,X)=A^{\partial \Sigma}(X,X)$, $\forall\ X\subset T_{p}(\partial \Sigma)$.  
    
\noindent Here $\eta$ is the outward unit normal vector of $\partial M\hookrightarrow M$.
\end{lemma}

\begin{proof}
    (1) This is trivial since $(\Sigma,\partial \Sigma)\hookrightarrow (M,\partial M)$ is a free boundary hypersurface.

    (2) $\forall\ X\in T_{p}(\partial \Sigma)$, we have
    \[
    \begin{aligned}
        A^{\partial \Sigma}(X,X)
        =\langle \nabla_{X}^{\Sigma}\eta_\Sigma,X\rangle
        =\langle \nabla_{X}^{M}\eta,X\rangle
        =A^{\partial M}(X,X).
    \end{aligned}
    \]
    The second equality is due to $\langle X,\nu\rangle =0$.
\end{proof}
Thus, for a stable weighted $k$-slicing with boundary, we simply denote $\eta_{\Sigma_{j}}=\eta$ for $ 1\leq j\leq k.$

We recall \textit{$m$-intermediate curvature} of $M$ in \cite{brendlegeroch'sconjecture}. We also recall the $m$-convexity of $\partial M$.
\begin{definition}

    Suppose $(M,g)$ is an $n$-dimensional manifold. For given orthonormal vectors $\{e_{1},...,e_{m}\}$ in $T_pM$ at the point $p\in M$, one can extend them to an orthonormal basis $\{e_{1},...,e_{n}\}$ of $T_{p}M$. 
    The $m$-intermediate curvature $C_{m}$ of the orthornormal vectors $\{e_{1},...,e_{m}\}$ is defined by
    \[
    C_{m}(e_{1},...,e_{m}) = \sum_{p=1}^{m}\sum_{q=p+1}^{n}Rm(e_{p},e_{q},e_{p},e_{q}).\]
In particular, $C_1(e_1)=\Ric(e_1,e_1)$ and $C_{n}(e_1,\cdots,e_{n})=2R.$ Let 
\[C_m(p):=\min\{C_m(e_1,\cdots,e_m)| \{e_1,\cdots,e_m\}\ \text{are orthonormal in}\ T_pM\}.\]
We say $(M,g)$ has nonnegative $m$-intermediate curvature if $C_m(p)\geq 0$ for all $p\in M.$
\end{definition}

\begin{definition}
Suppose $(M,\partial M,g)$ is $n$-dimensional manifold with boundary. For given orthonormal vectors $\{e_1,\cdots,e_m\}$ in $T_p\partial M$ at the point $p\in\partial M$, one can extend them to an orthonormal basis $\{e_1,\cdots,e_{n-1}\}$ of $T_p\partial M$. 
The $m$-intermediate boundary curvature $H_{m}$ of the orthornormal vectors $\{e_{1},...,e_{m}\}$ is defined by
\[ H_{m}(e_{1},...,e_{m}) = \sum_{i=1}^{m}A^{\partial M}(e_{i},e_{i}).
    \]
    In particular, $H_{n-1}(e_1,\cdots, e_{n-1})=H_{\partial M}$ at $p\in \partial M$. Let 
    \[H_m(p)=\min\{H_m(e_1,\cdots,e_m)| \{e_1,\cdots,e_m\}\ \text{are orthonormal in}\ T_p\partial M\}.\]
    We say $(M,g)$ has (strictly) $m$-convex boundary if $H_m(p)\geq 0$ $(>0)$ for all $p\in M.$
\end{definition}

\section{Proof of Theorem \ref{main theorem 1 Introduction}}\label{section 3}

This section is devoted to the proof of the first main result.

\begin{proof}[The proof of (1) in Theorem \ref{main theorem 1 Introduction}]
    According to Lemma \ref{existence of slicing}, there exists a stable weighted slicing with boundary of order $m$
    \[
    \Sigma_{m}\subset\cdots\subset \Sigma_{1}\subset\Sigma_{0}=M.
    \]
Denote the unit normal vector of $\Sigma_j\subset \Sigma_{j-1}$ by $\nu_{j}$.     For any $f\in C^{2}(\Sigma_{m})$, we have
    \[
    \begin{aligned}
        &\int_{\Sigma_{m}}\Big[-f\Delta_{\Sigma_{m}}f-(|A^{\Sigma_{m}}|^2+\text{Ric}_{\Sigma_{m-1}}(\nu_{m},\nu_{m}))f^2+(\nabla_{\Sigma_{m-1}}^{2}\log \rho_{m-1})(\nu_{m},\nu_{m})f^2\\
        &-f\langle \nabla_{\Sigma_{m}}\log \rho_{m-1},\nabla_{\Sigma_{m}}f\rangle\Big]\rho_{m-1} \ d\mu
        +\int_{\partial \Sigma_{m}}f(\frac{\partial f}{\partial \eta}-fA^{\partial \Sigma_{m-1}}(\nu_{m},\nu_{m})) dA\geq 0.
    \end{aligned}
    \]
    Taking $f=\rho_{m-1}^{-1}$. Then
    \begin{equation}\label{bottom term}
    \begin{aligned}
        &\int_{\Sigma_{m}}\rho_{m-1}^{-1}\big(\Delta_{\Sigma_{m}}\log \rho_{m-1} +(\nabla_{\Sigma_{m-1}}^{2}\log \rho_{m-1})(\nu_{m},\nu_{m})-|A^{\Sigma_{m}}|^2\\
        &-\text{Ric}_{\Sigma_{m-1}}(\nu_{m},\nu_{m})\big
        )\ d\mu -\int_{\partial \Sigma_{m}}\rho_{m-1}^{-1}(\frac{\partial \log \rho_{m-1}}{\partial \eta}+A^{\partial \Sigma_{m-1}}(\nu_{m},\nu_{m}))dA\geq 0.
    \end{aligned}
    \end{equation}
    The integration over $\Sigma_{m}$ has been
    estimated in \cite[Lemma 3.1]{brendlegeroch'sconjecture}. When $n\leq 7$, we have
    \[
    \begin{aligned}
        &\int_{\Sigma_{m}}\rho_{m-1}^{-1}\big(\Delta_{\Sigma_{m}}\log \rho_{m-1} +(\nabla_{\Sigma_{m-1}}^{2}\log \rho_{m-1})(\nu_{m},\nu_{m})-|A^{\Sigma_{m}}|^2\\
        &-\text{Ric}_{\Sigma_{m-1}}(\nu_{m},\nu_{m})\big
        )\ d\mu \leq \int_{\Sigma_{m}}-\rho_{m-1}^{-1}C_{m}(\nu_{1},...,\nu_{m})\ d\mu.
    \end{aligned}
    \]
    Now we estimate the integration over $\partial \Sigma_{m}$. Using the fact that $\frac{\rho_{j+1}}{\rho_{j}}$ is the first eigenfunction on $\Sigma_j$  and Lemma \ref{boundary}, we have
    \begin{equation}\label{eq1}
    \begin{aligned}
        \frac{\partial \log \rho_{m-1}}{\partial \eta}+A^{\partial \Sigma_{m-1}}(\nu_{m},\nu_{m})
        &=
        \frac{\partial \log \rho_{0}}{\partial \eta}+\sum_{i=0}^{m-1}A^{\partial \Sigma_{i}}(\nu_{i+1},\nu_{i+1})\\
        &=\sum_{i=1}^{m}A^{\partial M}(\nu_{i},\nu_{i})\\
        &=H_{m}(\nu_{1},...,\nu_{m}).
    \end{aligned}
    \end{equation}
    Thus, we have
    \begin{equation}\label{inequa1}
    \int_{\Sigma_{m}}\rho_{m-1}^{-1}C_{m}(\nu_{1},...,\nu_{m})\ d\mu + \int_{\partial \Sigma_{m}}\rho_{m-1}^{-1}H_{m}(\nu_{1},...,\nu_{m})\leq 0.
    \end{equation}
    Since $\rho_{m-1}>0$, we finish the proof of (1).
\end{proof}
In what follows we prepare to prove the second part of Theorem \ref{main theorem 1 Introduction}. If the inequality in the above proof holds, we have the following proposition.
\begin{proposition}\label{rigidity}
    Under the assumption of (2) of Theorem \ref{main theorem 1 Introduction}, the following is true:

    $(1)$ $\lambda_{k}=0$ for all $ 1\leq k\leq m-1$;

    $(2)$ $\nabla_{\Sigma_{m-1}}\log\rho_{m-1}=0$ along $\Sigma_{m}$;

    $(3)$ $\text{Ric}_{\Sigma_{m-1}}(\nu_{m},\nu_{m})=(\nabla^{2}_{\Sigma_{m-1}}\log \rho_{m-1})(\nu_{m},\nu_{m})$ on $\Sigma_{m}$;

    $(4)$ $C_{m}(\nu_{1},...,\nu_{m})=0$;

    $(5)$ $R_{\Sigma_{m}}=R_{M}|_{\Sigma_{m}}$;

    $(6)$ $A^{\Sigma_{k}}=0$ for all $\ 1\leq k\leq m$ on $\Sigma_{m}$;

    $(7)$ $\langle \nu_{k},\eta\rangle =0$ for all $\ 1\leq k\leq m$ along $\partial \Sigma_m$;
    
    $(8)$ $A^{\partial \Sigma_{m-1}}(\nu_{m},\nu_{m})=0$;

    $(9)$ $H_{m}(\nu_{1},...,\nu_{m})=0$ in $\partial\Sigma_{m-1}\subset \Sigma_{m-1}$.
\end{proposition}
\begin{proof}
     The proof of $(1)-(6)$ is given in \cite[Proposition 2.1]{ChuKwongLeerigidity}. Now we check $(7)-(9)$. $(7)$ is due to the fact that $\Sigma_{m}$ has free boundary.
   $ (9)$ follows from the inequality \eqref{inequa1}. By $(2)$, $\rho_{m-1}$ is constant along $\Sigma_{m}$. Then using equation \eqref{eq1}, we obtain $(8)$. 
\end{proof}
We use the inverse function theorem to construct a local free boundary foliation. It can be viewed as free boundary case of Lemma 3.1 in \cite{ChuKwongLeerigidity}. Readers are also recommended to see \cite{lucasrigiditythree} for a similar result in three dimension.
\begin{lemma}\label{foliation}
    If $\Sigma_{m}\subset \Sigma_{m-1}$ satisfies assumptions in Proposition \ref{rigidity}, then we can find a local foliation $\{\Sigma_{m,t}\}_{-\epsilon<t<\epsilon}$ of $\Sigma_{m}$ in $\Sigma_{m-1}$ such that $\Sigma_{m,t}$ is given by the graph over $\Sigma_m$ with graph function $u_t$ along the unit normal $\nu_m$. In other words, we can define a map $f_{t}:\Sigma_{m}\rightarrow \Sigma_{m-1}$ by $f_{t}(x)=\exp_{x}(u_{t}\nu_{m})$ and $\Sigma_{m,t} = f_{t}(\Sigma_{m})$. Moreover, it satisfies $$\Sigma_{m,0}=\Sigma_{m}, \left.\frac{\partial}{\partial t}\right|_{t=0}u_{t}=1, \frac{\partial u_{t}}{\partial t}>0, \int_{\Sigma_{m}}(u_{t}-t)=0, \langle \nu_{m,t},\eta_{t}\rangle=0$$ and $H_{\Sigma_{m,t}}+\langle \nabla_{\Sigma_{m-1}}\log \rho_{m-1},\nu_{m,t}\rangle = \text{const}$ on $\Sigma_{m,t}$ where $\nu_{m,t}$ is the unit normal vector field of $\Sigma_{m,t}$ and $\eta
    _{t}$ is the unit normal vector field of $\partial \Sigma_{m-1}\hookrightarrow \Sigma_{m-1}$ restricted to $\partial \Sigma_{m,t}$.
\end{lemma}
\begin{proof}
    Consider the Banach space $\mathcal{W}=\{(f,g)\in C^{\alpha}(\Sigma_{m})\times C^{\alpha}(\partial \Sigma_{m})|\int_{\Sigma_{m}}f-\int_{\partial \Sigma_{m}}g=0\}$. Given a function $u\in C^{2,\alpha}(\Sigma_{m})$, let $H_{\Sigma_u}$ denote the mean curvature of the graph $\Sigma_u$. Denote $\tilde{H}_{\Sigma_{u}}=H_{\Sigma_{u}}+\langle \nabla_{\Sigma_{m-1}}\log\rho_{m-1},\nu_{u}\rangle$ where $\nu_{u}$ is the unit normal vector field of $\Sigma_{u}$ and $\eta
    _{u}$ is the unit normal vector field of $\partial \Sigma_{m-1}\hookrightarrow \Sigma_{m-1}$ restricted to $\partial \Sigma_{u}$. Notice that $\Sigma_{m}$ is free boundary, then $\eta:=\eta_{u}|_{u=0} = \eta_{\Sigma_{m}}$, where $\eta_{\Sigma_{m}}$ is the outward conormal vector field of $\Sigma_{m-1}$. We can define a map $\Psi:C^{2,\alpha}(\Sigma_{m})\rightarrow\mathcal{W}\times \mathbb{R}$ as the following
    \[
    \Psi(u)=\Big(\big(\tilde{H}_{\Sigma_{u}}-\frac{1}{|\Sigma_{m}|}\int_{\Sigma_{m}}\tilde{H}_{\Sigma_{u}}+\frac{1}{|\Sigma_{m}|}\int_{\partial\Sigma_{m}}\langle \nu_{u},\eta_{u}\rangle,\langle\nu_{u},\eta_{u}\rangle\big),\frac{1}{|\Sigma_{m}|}\int_{\Sigma_{m}}u\Big).
    \]
    It is easy to check that $\Psi(0)=((0,0),0)$. Now we want to compute $D\Psi|_{u=0}$. We notice that
    \[
    \begin{aligned}
    &D\tilde{H}_{\Sigma_{u}}|_{u=0}(w)\\
    =&-\Delta_{\Sigma_{m}}w-|A^{\Sigma_{m}}|^2w-\text{Ric}_{\Sigma_{m-1}}(\nu_{m},\nu_{m})w+(\nabla^{2}_{\Sigma_{m-1}}\log \rho_{m-1})(\nu_{m},\nu_{m})w\\
    =&-\Delta_{\Sigma_{m}}w
    \end{aligned}
    \]
    where we applied $(3)$ and $(6)$ in Proposition \ref{rigidity}. 
    We also compute 
    \[
    \begin{aligned}
    D\langle\nu_{u},\eta_{u}\rangle|_{u=0}(w) 
    &=\langle -\nabla_{\Sigma_{m-1}}w,\eta\rangle + w\langle \nu_{m},\nabla_{\nu_{m}}\eta\rangle\\
    &=-\frac{\partial w}{\partial \eta}+ wA^{\partial \Sigma_{m-1}}(\nu_{m},\nu_{m})\\
    &=-\frac{\partial w}{\partial \eta}
    \end{aligned}
    \]
    where the final equality is due to $A^{\partial \Sigma_{m-1}}(\nu_{m},\nu_{m})=0$ from Proposition \ref{rigidity}. Thus, we have
    \[
    D\Psi|_{u=0}(w) = \Big(\big(-\Delta_{\Sigma_{m}}w,-\frac{\partial w}{\partial\eta}\big),\frac{1}{|\Sigma_{m}|}\int_{\Sigma_{m}}w\Big)
    \]
    where we used the fact that $\int_{\Sigma_{m}}\Delta_{\Sigma_{m}}w=\int_{\partial \Sigma_{m}}\frac{\partial w}{\partial \eta}$ and $\eta = \eta_{\Sigma_{m}}$.
    By the uniqueness and existence of solutions to the elliptic equation under Nuemann boundary condition, we obtain that $D\Psi|_{u=0}$ is injective and surjective. According to the inverse function theorem, there exists a small $\epsilon>0$ such that for any $t\in (-\epsilon,\epsilon)$, there exists $u_{t}$ satisfying $\Psi(u_{t})=((0,0),t)$. Differentiating the equation $\Psi(u_{t})=((0,0),t)$, we have
    \[
    \Big(\big(-\Delta_{\Sigma_{m}}(\left.\frac{\partial }{\partial t}\right|_{t=0}u_{t}),-\frac{\partial}{\partial\eta}(\left.\frac{\partial}{\partial t}\right|_{t=0}u_{t})\big),\frac{1}{|\Sigma_{m}|}\int_{\Sigma_{m}}(\left.\frac{\partial}{\partial t}\right|_{t=0}u_{t})\Big)=(0,0,1).
    \]
    Then $\left.\frac{\partial}{\partial t}\right|_{t=0}u_{t}=1$. Taking $\epsilon$ smaller, we get $\partial_{t}u_{t}>0$. 
\end{proof}

Finally, we use a classical argument to show that the foliation constructed above consists of free boundary minimizing hypersurfaces (see \cite{XuTrans}).

\begin{proposition}\label{foliation2}
    Let $\Sigma_{m,t}$ be constructed as above. If $n\leq 6$, $\Sigma_{m,t}$ is a free boundary $\mathcal{H}_{\rho_{m-1}}$-minimizing hypersurface. Moreover, for any $ -\epsilon<t<\epsilon$, we have
    \begin{equation}\label{eqfolaition}
        \int_{\Sigma_{m}}\rho_{m-1}d\mu = \int_{\Sigma_{m,t}}\rho_{m-1}d\mu.
    \end{equation}
\end{proposition}
\begin{proof}
    In the following we only consider the case $t>0$. The argument on  the case $t<0$ is similar. 
    It follows from the first variation formula that
    \[
    \begin{aligned}
    &\int_{\Sigma_{m,t}}\rho_{m-1}\ d\mu -\int_{\Sigma_{m}}\rho_{m-1}\ d\mu\\=
    &\int_{0}^{t}\int_{\Sigma_{m,s}}\rho_{m-1}\frac{\partial u_{s}}{\partial s}\tilde{H}_{\Sigma_{m,s}}d\mu ds+\int_{0}^{t}\int_{\partial \Sigma_{m,s}}\rho_{m-1}\frac{\partial u_{s}}{\partial s}\langle \nu_{s},\eta_{s}\rangle dA ds\\=&\int_{0}^{t}\int_{\Sigma_{m,s}}\rho_{m-1}\frac{\partial u_{s}}{\partial s}\tilde{H}_{\Sigma_{m,s}}d\mu ds.
    \end{aligned}
    \]
   The second equality holds due to the fact that each $\Sigma_{m,s}$ is free boundary. We claim that $\tilde{H}_{\Sigma_{m,t}}\leq 0$ for $t>0$. Then we have $\int_{\Sigma_{m,t}}\rho_{m-1}\ d\mu \leq \int_{\Sigma_{m}}\rho_{m-1}\ d\mu$. Since $\Sigma_{m}$ is $\mathcal{H}_{\rho_{m-1}}$-minimizing, we obtain that $\Sigma_{m,t}$ is also $\mathcal{H}_{\rho_{m-1}}$-minimizing hypersurface, completing the proof of the Proposition.
   
   In what follows  we prove the claim. Denote $\varphi_{t}\nu_{m,t}$ the variational vector field, we directly compute $\frac{d \tilde{H}_{\Sigma_{m,t}}}{d t}$. Note that
   \[
   \begin{aligned}
   \frac{d\tilde{H}_{\Sigma_{m,t}}}{dt}
   &=-\Delta_{\Sigma_{m,t}}\varphi_{t}-(|A^{\Sigma_{m,t}}|^2+\text{Ric}_{\Sigma_{m-1}}(\nu_{m,t},\nu_{m,t}))\varphi_{t}\\
   &+\nabla_{\Sigma_{m-1}}^{2}\log\rho_{m-1}(\nu_{m,t},\nu_{m,t})\varphi_{t}-\langle \nabla_{\Sigma_{m,t}}\varphi_{t},\nabla_{\Sigma_{m,t}}\log\rho_{m-1}\rangle.
   \end{aligned}
   \]
Multiplying both sides by $\varphi_{t}^{-1}$ and integrating over $\Sigma_{m,t}$, we have
\[
\begin{aligned}
    &\frac{d\tilde{H}_{\Sigma_{m,t}}}{dt}\int_{\Sigma_{m,t}}\varphi_{t}^{-1}d\mu\\
   &=\int_{\Sigma_{m,t}}\Big[-\varphi_{t}^{-1}\Delta_{\Sigma_{m,t}}\varphi_{t}-(|A^{\Sigma_{m,t}}|^2+\text{Ric}_{\Sigma_{m-1}}(\nu_{m,t},\nu_{m,t}))\\
   &+\nabla_{\Sigma_{m-1}}^{2}\log\rho_{m-1}(\nu_{m,t},\nu_{m,t})-\langle \nabla_{\Sigma_{m,t}}\log\varphi_{t},\nabla_{\Sigma_{m,t}}\log\rho_{m-1}\rangle\Big]d\mu.
\end{aligned}
\]
Let $\tilde{\rho}_{m}=\varphi_{t}\rho_{m-1}$, using integration by part yields
\[
\begin{aligned}
    &\int_{\Sigma_{m,t}}\Big[-\varphi_{t}^{-1}\Delta_{\Sigma_{m,t}}\varphi_{t}-\langle \nabla_{\Sigma_{m,t}}\log\varphi_{t},\nabla_{\Sigma_{m,t}}\log\rho_{m-1}\rangle\Big]d\mu\\
    &=-\int_{\Sigma_{m,t}}\langle \nabla_{\Sigma_{m,t}}\log\varphi_{t},\nabla_{\Sigma_{m,t}}\log\tilde{\rho}_{m}\rangle d\mu-\int_{\partial\Sigma_{m,t}}\varphi_{t}^{-1}\frac{\partial \varphi_{t}}{\partial \eta_{t}}dA.
\end{aligned}
\]
and
\[
\begin{aligned}
    &\int_{\Sigma_{m,t}}\nabla_{\Sigma_{m-1}}^{2}\log\rho_{m-1}(\nu_{m,t},\nu_{m,t})d\mu\\
    &=\int_{\Sigma_{m,t}}\Big(\Delta_{\Sigma_{m-1}}\log\rho_{m-1}-\Delta_{\Sigma_{m,t}}\log\rho_{m-1}-H_{\Sigma_{m,t}}\langle \nabla_{\Sigma_{m-1}}\log\rho_{m-1},\nu_{m,t}\rangle\Big)d\mu \\
    &=\int_{\Sigma_{m,t}}\Big(\Delta_{\Sigma_{m-1}}\log\rho_{m-1}-H_{\Sigma_{m,t}}\langle \nabla_{\Sigma_{m-1}}\log\rho_{m-1},\nu_{m,t}\rangle \Big)d\mu-\int_{\partial \Sigma_{m,t}}\frac{\partial \log\rho_{m-1}}{\partial \eta_{t}}dA.
\end{aligned}
\]
Combining all we obtain
\[
\begin{aligned}
    &\frac{d\tilde{H}_{\Sigma_{m,t}}}{dt}\int_{\Sigma_{m,t}}\varphi_{t}^{-1}d\mu\\
   &=\int_{\Sigma_{m,t}}\Big[-\langle \nabla_{\Sigma_{m,t}}\log\varphi_{t},\nabla_{\Sigma_{m,t}}\log\tilde{\rho}_{m}\rangle-|A^{\Sigma_{m,t}}|^2-\text{Ric}_{\Sigma_{m-1}}(\nu_{m,t},\nu_{m,t})
   \\
   &+\Delta_{\Sigma_{m-1}}\log\rho_{m-1}-H_{\Sigma_{m,t}}\langle \nabla_{\Sigma_{m-1}}\log\rho_{m-1},\nu_{m,t}\rangle\Big]d\mu\\
   &-\int_{\partial \Sigma_{m,t}}\Big(\frac{\partial \log\rho_{m-1}}{\partial \eta_{t}}+\varphi_{t}^{-1}\frac{\partial \varphi_{t}}{\partial \eta_{t}}\Big)dA.
\end{aligned}
\]
The integrand over $\Sigma_{m,t}$ is estimated in \cite{XuTrans}. The following is taken from \cite[page 2109 - page 2110]{XuTrans}.
\[
\begin{aligned}
    &\int_{\Sigma_{m,t}}\Big[-\langle \nabla_{\Sigma_{m,t}}\log\varphi_{t},\nabla_{\Sigma_{m,t}}\log\tilde{\rho}_{m}\rangle-|A^{\Sigma_{m,t}}|^2-\text{Ric}_{\Sigma_{m-1}}(\nu_{m,t},\nu_{m,t})
   \\
   &+\Delta_{\Sigma_{m-1}}\log\rho_{m-1}-H_{\Sigma_{m,t}}\langle \nabla_{\Sigma_{m-1}}\log\rho_{m-1},\nu_{m,t}\rangle\Big]d\mu\\
   &\leq \int_{\Sigma_{m,t}}\Big[-\frac{H_{\Sigma_{m,t}}^2}{n-m}-H_{\Sigma_{m,t}}\langle\nabla_{\Sigma_{m-1}}\log\rho_{m-1},\nu_{m}\rangle-\frac{m}{2(m-1)}\langle\nabla_{\Sigma_{m-1}}\log\rho_{m-1},\nu_{m}\rangle^2\\
   &-\frac{m}{2(m-1)}|\nabla_{\Sigma_{m}}\log\rho_{m-1}|^2-\langle\nabla_{\Sigma_{m}}\log\varphi_{t},\nabla_{\Sigma_{m}}(\log\rho_{m-1}+\log\varphi_{t})\rangle\Big]d\mu.
\end{aligned}
\]
From this to be non-positive, we need
\[
\frac{2m}{(n-m)(m-1)}\geq1
\]
which is satisfied for $n\leq 6$ and $m\leq n-1$ or $n=7$ and $m\in \{1,5,6\}$.
Now, we estimate the integration over $\partial \Sigma_{m,t}$. Since $\langle\nu_{t},\eta_{t}\rangle=0$, we get
\[
0=\frac{d}{dt}\langle\nu_{t},\eta_{t}\rangle=-\frac{\partial\varphi_{t}}{\partial \eta_{t}}+\varphi_{t}A^{\partial M}(\nu_{m,t},\nu_{m,t}).
\]
Then
\[
\begin{aligned}
\int_{\partial \Sigma_{m,t}}\Big(\frac{\partial \log\rho_{m-1}}{\partial \eta_{t}}+\varphi_{t}^{-1}\frac{\partial \varphi_{t}}{\partial \eta_{t}}\Big)dA
&=\int_{\partial \Sigma_{m,t}}\Big(\frac{\partial \log\rho_{m-1}}{\partial \eta_{t}}+A^{\partial M}(\nu_{m,t},\nu_{m,t})\Big)dA\\
&=\int_{\partial \Sigma_{m,t}}H_{m}^{\partial M}(\nu_{1},...,\nu_{m-1},\nu_{m,t})dA\geq0. 
\end{aligned}
\]
Hence, $\frac{d\tilde{H}_{\Sigma_{m,t}}}{dt}\leq0$ and thus $\tilde{H}_{\Sigma_{m,t}}\leq 0$ as claimed.
\end{proof}

Now we are ready to prove the second part in Theorem \ref{main theorem 1 Introduction}.
\begin{proof}[The proof of (2) in Theorem \ref{main theorem 1 Introduction}]
    Using Proposition \ref{rigidity}, Lemma \ref{foliation} and Proposition \ref{foliation2}, there exists a local foliation $\{\Sigma_{m,t}\}$ of $\Sigma_{m}$ such that $\Sigma_{m,t}$ is free boundary weighted area-minimizing hypersurface in $\Sigma_{m-1}$. We denote the metric on $\Sigma_{m-1}$ and $\Sigma_{m,t}$ by $g_{m-1}$ and $g_{m,t}$. For each $t\in(-\epsilon,\epsilon)$, we can define the lapse function $h_{t} = \langle \nu_{m,t},\frac{\partial f_{t}}{\partial t}\rangle$. Then we can write $g_{m-1}=h_{t}^2dt^{2}+g_{m,t}$ locally. Since $\Sigma_{m,t}$ is free boundary weighted area-minimizing hypersurface in $\Sigma_{m-1}$, then $\Sigma_{m,t}$ also satisfies Proposition \ref{rigidity}. Then $\nabla_{\Sigma_{m-1}}\log\rho_{m-1}=0$ along $\Sigma_{m,t}$, which implies that $\rho_{m-1}$ is constant on $\{\Sigma_{m,t}\}_{-\epsilon<t<\epsilon}$. 
    According to Proposition \ref{foliation2}, we have
    \[
    \frac{d^2}{dt^2}\int_{\Sigma_{m,t}}\rho_{m-1}d\mu =0.
    \]
    Thus, the lapse function $h_{t}$ satisfy the Jacobi equation
    \[
    \Delta_{\Sigma_{m,t}}h_{t}+(\text{Ric}_{\Sigma_{m-1}}(\nu_{m,t},\nu_{m,t})+|A^{\Sigma_{m,t}}|^2)h_{t}=0.
    \]
    Notice that $\text{Ric}_{\Sigma_{m-1}}(\nu_{m,t},\nu_{m,t})=\nabla_{\Sigma_{m-1}}^{2}\log\rho_{m-1}(\nu_{m,t},\nu_{m,t})=0$. Moreover, Proposition \ref{rigidity} implies that $\Sigma_{m,t}$ is totally geodesic and $A^{\Sigma_{m,t}}=0$. Thus, $h_{t}$ satisfies 
    \[
        \Delta_{\Sigma_{m,t}}h_{t}=0.    
    \]
    Differentiating $\langle \nu_{m,t},\eta_{m,t}\rangle = 0$ along $h_{t}\nu_{m,t}$, we have $$\frac{\partial h_{t}}{\partial \eta}=h_{t}A^{\partial M}(\nu_{m,t},\nu_{m,t})=0.$$ Then we get $h_{t}$ is a constant and thus $h_{t}(x)=\phi(t)$. Notice $\Sigma_{m,t}$ is totally geodesic, then $\partial_{t}g_{m,t}=0$, which implies $g_{m-1}=\phi(t)^2dt^2+g_{m}$ locally. After reparametrizing $t$ by letting $s(t)=\int_{0}^{t}\phi(\tau)d\tau$, we get $g_{m-1}=ds^2+g_{m}$ in $\Sigma_{m}\times (-\epsilon,\epsilon)$. Using the continuity argument as in \cite{BrayBrenleNevesrigidity}, we conclude that $\Sigma_{m-1}$ is isometrically covered by $\Sigma_{m}\times \mathbb{R}$ and $\rho_{m-1}$ is constant on $\Sigma_{m-1}$. By same argument in \cite{ChuKwongLeerigidity}, we obtain
    \[
    C_{m-1}(\nu_{1},...,\nu_{m-1})=0.
    \]
    On the boundary, we have
    \[
        H_{m-1}(\nu_{1},...,\nu_{m-1})=H_{m}(\nu_{1},...,\nu_{m})-A^{\partial M}(\nu_{m},\nu_{m})=0.
    \]
    Repeating the argument inductively, we obtain that for any $k=1,\cdots,m$, $C_{k}(\nu_{1},...,\nu_{k})=0$, $H_{k}(\nu_{1},...,\nu_{k})=0$ and $\Sigma_{k-1}$ is isometrically covered by $\Sigma_{k}\times \mathbb{R}$. Hence, $M$ is isometrically covered by $\Sigma_{m}\times \mathbb{R}^{m}$. $\text{Ric}_{\Sigma_{m}}\geq 0$ is showed in \cite{ChuKwongLeerigidity}. Next, we show that $A^{\partial \Sigma_{m}}\geq 0$. In fact, for any $ X\in T_{p}(\partial \Sigma_{m})$, we have
    \[
    0\leq H_{m}(X,\nu_{1},...,\nu_{m-1})=A^{\partial M}(X,X)=A^{\partial \Sigma_{m}}(X,X).
    \]
    The proof is complete.
\end{proof}
To conclude this section, we show that the topological condition is preserved under connected sums with a compact manifold.

\begin{definition}\label{interior connected sum}
Let $M_1$ and $M_2$ be two $n$-manifolds with non-empty boundaries. The \textit{interior connected sum} $M_1\#_i M_2$ is constructed by removing the interiors of two standard balls $B_i^n \subset \operatorname{int}(M_i)$ and gluing the resulting boundary spheres $\partial B_i^n$ via an orientation-reversing diffeomorphism.
\end{definition}

\begin{definition}\label{boundary connected sum}
Let $M_1$ and $M_2$ be two $n$-manifolds with nonempty boundaries. Let $\partial_1 M_1$ and $\partial_2 M_2$ be two specific boundary components of $M_1$ and $M_2$, respectively. The \textit{boundary connected sum} $M_1\#_b M_2$ (with respect to $\partial_1M_1$ and $\partial_2 M_2$) is constructed by identifying two standard balls $B_i^{n-1}\subset \partial_i M_i$ via an orientation-reversing diffeomorphism.
\end{definition}

\begin{remark}
    (1) The boundary connected sum of two handbodies is again a handlebody;
    
    (2) The boundary connected sum of two punctured handlebodies is a punctured handlebody (specifically, with the number of punctures being the sum from each summand).
\end{remark}

\begin{corollary}\label{cor3}
    Let $3 \leq n \leq 7$, and let $(M, \partial M)$ be an $n$-dimensional compact orientable manifold satisfying the conditions of Theorem \ref{main theorem 1 Introduction}. For any $n$-dimensional compact orientable manifold $N$, the conclusion of Theorem \ref{main theorem 1 Introduction} holds for the manifold $M \# N$.
\end{corollary}

\begin{proof}
    Consider the projection map $P: M \# N \to M$ that collapses $N \setminus \{\text{ball}\}$ to a point.  Via this map, we can pull back the cohomology classes $\alpha^1, \dots, \alpha^m$ from $M$ to obtain classes $P^*\alpha^1, \dots, P^*\alpha^m \in H^1(M \# N; \mathbb{Z})$. To apply Theorem \ref{main theorem 1 Introduction} to $M \# N$, we must verify that the cup product $P^*\alpha^1 \smile \cdots \smile P^*\alpha^m$ is non-zero both on $M \# N$ and on its restriction to the boundary $\partial(M \# N)$.

    We analyze the two types of connected sums separately.

    \textbf{Case 1: Interior Connected Sum ($M \#_i N$).}
    In this case, the boundary satisfies $\partial M\subset \partial(M \#_i N)$. Since the projection $P$ restricts to the identity on $\partial M$, the pullback $P^*$ induces the identity map on cohomology. By the assumption on $M$, we have $\alpha^1 \smile \cdots \smile \alpha^m|_{\partial M} \neq 0$. Therefore, it follows immediately that
    \[
    (P^*\alpha^1 \smile \cdots \smile P^*\alpha^m)|_{\partial(M \#_i N)} \neq 0.
    \]

    \textbf{Case 2: Boundary Connected Sum ($M \#_b N$).}
    WLOG, we can assume $\partial N$ is connected. Otherwise, we can replace $\partial N$ by the connected component which connects sum with $\partial M$. By definition, the boundary of the boundary connected sum is given by $\partial(M \#_b N) = \partial M \#_i \partial N$. 

     In fact, consider the pair $(\partial M \backslash D^{n-1},S^{n-2})\hookrightarrow(\partial M\#_{i}\partial N,\partial N\backslash D^{n-1})$, we have the following commutative diagram
    \[
    \begin{tikzcd}
    H^{m-1}(\partial N\backslash D^{n-1}) \arrow[r, "\delta"] \arrow[d, swap, "\iota^{*}"] & H^{m}(\partial M\#_{i}\partial N,\partial N\backslash D^{n-1}) \arrow[d, "\iota^{*}"] \\
    H^{m-1}(S^{n-2}) \arrow[r, swap, "\delta"] & H^{m}(\partial M \backslash D^{n-1},S^{n-2})
    \end{tikzcd}
    \]
    By excision lemma $H^{m}(\partial M\#_{i}\partial N,\partial N\backslash D^{n-1})=H^{m}(\partial M \backslash D^{n-1},S^{n-2})$ and $H^{m-1}(S^{n-2})=0$ for $1\leq m\leq n-2$, we obtain $\delta:H^{m-1}(\partial N\backslash D^{n-1})\rightarrow H^{m}(\partial M\#_{i}\partial N,\partial N\backslash D^{n-1})$ is zero map. By the long exact sequence,
    \[
     H^{m-1}(\partial N\backslash D^{n-1})\rightarrow H^{m}(\partial M\#_{i}\partial N,\partial N\backslash D^{n-1})\rightarrow H^{m}(\partial M\#_{i}\partial N).
    \]
    we obtain $P^{*}:H^{m}(\partial M\#_{i}\partial N,\partial N\backslash D^{n-1})\rightarrow H^{m}(\partial M\#_{i}\partial N)$ is injective. Notice that $H^{m}(\partial M) = H^{m}(\partial M\#_{i}\partial N,\partial N\backslash D^{n-1})$, we get $P^{*}$ is injective.
    Since $\partial N$ is orientable, in the top degree we have $P^*:H^{n-1}(\partial M) \cong H^{n-1}(\partial M \#_i \partial N)$. In conclusion, this induces an injection $P^*: H^*(\partial M; \mathbb{Z}) \hookrightarrow H^*(\partial(M \#_b N); \mathbb{Z})$ on cohomology. Therefore, since $\alpha^1 \smile \cdots \smile \alpha^m|_{\partial M} \neq 0$, we conclude that
    \[
    (P^*\alpha^1 \smile \cdots \smile P^*\alpha^m)|_{\partial(M \#_b N)} = P^*(\alpha^1 \smile \cdots \smile \alpha^m|_{\partial M}) \neq 0.
    \]

    In both cases, the necessary cohomological conditions are satisfied by $M \# N$. The non-vanishing of the cup product $P^*\alpha^1 \smile \cdots \smile P^*\alpha^m$ on the interior of $M \# N$ follows directly since it is non-vanishing on the boundary. The rest of proof can now be completed by following the same argument as in the proof of Theorem \ref{main theorem 1 Introduction}.
\end{proof}

 When the manifold we glue is nonorientable, we obtain the nonexistence result.
\begin{corollary}\label{cor2}
    Let $3\leq n\leq 7$, $1\leq m\leq n-1$ and let $(M,\partial M)$ be an $n$-dimensional compact orientable manifold. Suppose there exist nonzero $  \alpha^{1},...,\alpha^{m}\in H^{1}(M,\mathbb{Z})$ such that $\alpha^{1}\smile\cdots\smile \alpha^{m}\neq 0$. Then for any $n$-dimensional compact manifold $N$, the connected sum $M\# N$ does not admit a complete Riemannian metric $g$ of positive $m$-intermediate curvature and $m$-convex boundary.
\end{corollary}
\begin{proof}
    Similar to the proof of Corollary \ref{cor3}, we consider the map $P:M\#N\rightarrow M$. We want to show $P^*\alpha^{1}\smile \cdots \smile P^{*}\alpha^{m}\neq 0\in H^{m}(M\#N,\mathbb{Z})$. It suffices to prove that $P^*$ is injective. When $\#$ denotes the interior connected sum, the proof is similar to the proof of Corollary \ref{cor3}. We only show the case of the boundary connected sum.  
      

    In fact, consider the pair $(M,D^{n-1})\hookrightarrow(M\#_{b}N,N)$, we have the following commutative diagram
    \[
    \begin{tikzcd}
    H^{m-1}(N) \arrow[r, "\delta"] \arrow[d, swap, "\iota^{*}"] & H^{m}(M\#_{b}N,N) \arrow[d, "\iota^{*}"] \\
    H^{m-1}(D^{n-1}) \arrow[r, swap, "\delta"] & H^{m}(M,D^{n-1})
    \end{tikzcd}
    \]
 By excision lemma $H^{m-1}(M\#_{b}N,N)=H^{m}(M,D^{n-1})$ and $H^{m-1}(D^{n-1})=0$, we obtain $\delta:H^{m-1}(N)\rightarrow H^{m}(M\#_{b}N,N)$ is zero map. By the long exact sequence,
\[
    H^{m-1}(N)\rightarrow H^{m}(M\#_{b}N,N)\rightarrow H^{m}(M\#_{b}N).
\]
we obtain $P^{*}:H^{m}(M\#_{b}N,N)\rightarrow H^{m}(M\#_{b} N)$ is injective. Notice that $H^{m}(M\#_{b}N,N) = H^{m}(M)$, we get $P^{*}$ is injective.
\end{proof}
\section{Nonexistence under connected sum}\label{section 4}
In this section, we would like to study whether the manifold in Theorem \ref{main theorem 1 Introduction} with punctures still admit a metric of positive scalar curvature and mean curvature. In fact, we can prove the following. 

\begin{theorem}\label{main theorem 2}
    Assume either $3\leq n\leq 5$, $1\leq m\leq n-1$ or $6\leq n\leq 7$, $m\in\{1,n-2,n-1\}$. Let $(M,\partial M)$ be a compact orientable manifold with dimension $n$. Suppose that there exist nonzero $  \alpha^{1},...,\alpha^{m}\in H^{1}(M,\mathbb{Z})$ such that $\alpha^{1}\smile\cdots\smile \alpha^{m}\neq 0$. For any $n$-dimensional manifold $N$, the connected sum $M\# N$ does not admit a complete Riemannian metric $g$ of positive $m$-intermediate curvature $C_{m}$ and $m$-convex boundary.
\end{theorem}

The connect sum $M\# N$ means interior connect sum or boundary connect sum. When it refers to boundary connect sum, the boundary components we choose to glue do not affect the result. Thus we simply assume that connect sum is conducted at two arbitrary boundary components of $M$ and $N$, respectively.

If $N$ is a compact manifold, the result follows from Corollary \ref{cor2}. Thus we will always assume that $N$ is noncompact.

To prove this theorem, we first construct an infinite cyclic cover $\hat{M}$ of $M$ by cutting and pasting. According to Lemma \ref{represented}, let $\Sigma$ be a smooth embedded hypersurface with the boundary $\partial\Sigma\subset 
\partial M$ that represents the Lefschetz dual of $\alpha^{1}$. We can assume $\Sigma$ is connected. Otherwise, let $\Sigma=\cup_{i}\Sigma_{i}$, where $\Sigma_{i}$ are the connected components. Denote the Lefschetz dual of $[\Sigma_{i}]$ by $\gamma^{i}$, then we have $\alpha^{1}=\sum_{i}\gamma^{i}$.  There exists $\gamma_{i}$ such that $\gamma^{i}\smile(\smile_{j=2}^{m}\alpha^{j})\neq 0$, so we replace $\Sigma$ by $\Sigma_{i}$ and still denote it by $\Sigma$. It  follows from Lemma \ref{separate} that $\tilde{M}=M\backslash \Sigma$ is connected manifold with piecewise smooth boundary. 
We can assume that $\Sigma$ transversally intersects $M$ by Lemma \ref{represented}. Let $\tilde{M}_k, k\in \mathbb{Z}$ be infinite copies of $\tilde{M}$. We glue all $\tilde{M}_k$ along the portion of the boundary by gluing the $\Sigma_+$ part of $\tilde{M}_k$ with the $\Sigma_-$ boundary portion of $\tilde{M}_{k+1}$. Denote the resulting manifold by
\[\hat{M}=\cup_{k\in\mathbb{Z}}\tilde{M}_k/{\sim}\]
where the equivalence relation $\sim$ is the gluing. In fact, $\hat{M}$ is an infinite cyclic covering of $M$ which is a noncompact manifold with noncompact boundary.

Let $Y=M\# N$. Let $B_p(R)$ be a ball in $M$ with small radius $R$  such that $B_p(R)\cap \partial M =\emptyset$ and $B_p(R)\cap \Sigma=\emptyset$. Let $B$ be a small ball in $N$ such that $B\cap \partial N=\emptyset$. Denote $M'=M\setminus B_p(R)$ and $N'=N\setminus B$. Then we have that $Y=M'\cup N'$ where $M'$ and $N'$ are glued along the boundary of $B_p(R)$ and $B$ by using an oreintation-reversing diffeomorphism. 

Suppose that $Y$ admits a metric of positive $m$-intermediate curvature and $m$-intermediate boundary curvature, we can assume that the $m$-intermediate curvature satisfies $C_m>1>0$ on $M'$ and the $m$-intermediate boundary curvature satisfies $H_m>b>0$ on $\partial M.$

\subsection{Free boundary $\mu$-bubble}
Let $(M,\partial M,g)$ be an $n$-dimensional manifold with piecewise smooth boundary. Assume that $\partial M=\partial_+M\cup \partial_{-}M\cup \partial_0M$ where $\partial_+ M,\partial_{-}M,$ and $\partial_0M$ are smooth $(n-1)$-dimensional submanifolds (possibly with boundary). Assume that $\partial_+M\cap\partial_-M=\emptyset.$ We also assume that $\partial_+M,\partial_-M$ intersect $\partial_0M$ with non-obtuse angles along the intersection (if they meet). That is, the interior angles between $\partial_+M, \partial_{-}M$ and $\partial_0M$ are less than or equal to $\frac{\pi}{2}$. This will force the minimizer of the following functional $\mathcal{A}(\cdot)$ to stay away from the corner so that the minimizer is smooth. For any open set $\Omega\subset M$, denote the topological boundary $\partial\Omega$ by $\partial\Omega=\overline{\partial \Omega\cap M^\circ}.$

Let $h\in C^{\infty}(M\setminus (\partial_+M\cup \partial_-M))$ such that $h\rightarrow \pm\infty$ on $\partial_{\pm}M$. Let $\Omega_0$ be a reference Caccioppoli set containing $\partial_+M$ such that $\partial\Omega_0\subset M^\circ\cup \partial_0M.$ For example, let $c$ be the regular value of $h$ and denote $\Omega_0=h^{-1}((c,\infty)).$

Consider the functional
\[\mathcal{A}(\Omega)=|\partial\Omega|-\int_M(\chi_\Omega-\chi_{\Omega_0})h\]
for all Caccioppoli set $\Omega\subset M$ with $\Omega\Delta\Omega_0\subset M\setminus(\partial_+M\cup \partial_-M).$

\begin{proposition}[Chodosh-Li \cite{chodoshlisoapbubble}]\label{secondvariationof-A}
    There exists a minimizer $\Omega$ to the functional $\mathcal{A}(\Omega)$ with $\partial\Omega\subset M^\circ\cup \partial_0M.$ The boundary $\partial\Omega$ is smooth and meets $\partial_0M$ orthogonally. The critical equation is
    \[H=h\]
    along $\partial\Omega$ where $\nu$ is the outer unit normal of $\Omega$. On each component, $\Sigma$, of $\partial\Omega$, we have that for any $\varphi\in C^1(\Sigma)$,
    \begin{align*}
    0&\leq \int_{\Sigma} |\nabla_{\Sigma}\varphi|^2-(|A^{\Sigma}|^2+\text{Ric}_{M}(\nu,\nu)+\nabla_\nu h)\varphi^2-\int_{\partial\Sigma}A^{\partial M}(\nu,\nu)\varphi^2\\
    &=\int_{\Sigma} -\big(\Delta_{\Sigma}\varphi+(|A^{\Sigma}|^2+\text{Ric}_{M}(\nu,\nu)+\nabla_\nu h)\varphi\big)\varphi+\int_{\partial\Sigma}\varphi(\frac{\partial \varphi}{\partial \eta}-A^{\partial M}(\nu,\nu)\varphi).
    \end{align*}
\end{proposition}

Note that the boundary integral term disappears in the above inequality provided the boundary component $\Sigma$ is closed. 
\vskip.2cm
\subsection{Proof of Theorem \ref{main theorem 2}}
We first describe an infinite cyclic covering $\hat{Y}$ of $Y.$ 
Let $\tilde{M}^{'}_k$, $k\in\mathbb{Z}$ be infinite copies of $\tilde{M}^{'}$, we then glue $\tilde{M}_k^{'}$ along some portion of the boundary by gluing the $\Sigma_+$ part of $M_{k}^{'}$ with the $\Sigma_-$ boundary portion of $M^{'}_{k+1}.$ Denote
\[\hat{M}^{'}=\cup_{k\in\mathbb{Z}}\tilde{M}^{'}_k/{\sim},\]
where $\sim$ is the equivalence relation corresponding to the gluing we defined above.

Then we have
\[\hat{Y}=\hat{M}\#_\mathbb{Z} N=\hat{M}^{'}\cup (\cup_{k\in\mathbb{Z}}N^{'}_k)\]
where each $N^{'}_k$ are attached to $\hat{M}^{'}$ by gluing on the boundary spheres. Notice that there are infinitely many copies of $\Sigma$ in $\hat{Y}$, we denote the $\Sigma_+$ part of $\tilde{M}_0^{'}$ by $\Sigma_0.$ By our assumption, we can assume that the $m$-intermediate curvature satisfies $(C_{m})_{\hat{Y}}>1$ on $\hat{M}^{'}$ and the $m$-intermediate boundary curvature satisfies $H_{m}>b>0$ on $\partial \hat{M}.$ 

Fix $a>0$. Next, we can construct a weight function $h$ corresponding to $a$ satisfying 
\begin{equation}\label{equ about h}
(C_{m})_{\hat{Y}}+a h^2-2|D_{\hat{Y}}h|>0.
\end{equation}
For convenience of readers, we write the detail. In particular, we should consider how to construct $h$ near the boundary of $\hat{Y}$ such that the angle between $\partial(\{|h|<\infty\})\setminus \partial\hat{Y}$ and $\partial \hat{Y}$ is less than or equal to $\frac{\pi}{2}$.

\begin{lemma}\label{acute angle}
    Suppose $(M,\partial M)$ is a smooth Riemmanian manifold with non-empty boundary. Let $\rho:M\rightarrow \mathbb{R}$ be a smooth function with $\text{Lip}(\rho)\leq D$. For any $a>0$ and $\epsilon>0$, if $\partial(\{\rho(x)< a\})$ is a compact set, then there exist a small constant $\epsilon'>0$ and a smooth function $\tilde{\rho}$ on $M$ such that  the dihedral angle between $  \partial(\{\tilde{\rho}<a\})\setminus \partial M$ and $ \partial M\cap \partial(\{\tilde{\rho}<a\})$ is less than or equal to $\frac{\pi}{2}$, $\text{Lip}(\tilde{\rho})\leq D+\epsilon'$ and $\tilde{\rho}(x)=\rho(x)$ for $x\in \{\rho(x)\leq a-\epsilon\}$. 
\end{lemma}

\begin{proof}
    We only consider for $\rho^{-1}(a)$ and $\rho^{-1}(a-\epsilon)$ are smooth hypersurfaces which transversally intersect $\partial M$. We call the angle formed by $\partial(\{\rho(x)<a\})\cap\partial M$ and $\partial(\{\rho(x)<a\})\setminus \partial M$ the dihedral angle. If the dihedral angle is no greater than $\frac{\pi}{2}$, then the proof is done. Otherwise, we shall construct a new smooth hypersurface with the desired angle condition to replace $\rho^{-1}(a)$. The construction is inspired by \cite{chen2024positivescalarcurvaturemetrics}. Notice that $\rho^{-1}(a)$ is a compact hypersurface. Thus, there exists a tubular neighborhood  $B_{\epsilon_{1}}(\rho^{-1}(a))$ for $\epsilon_{1}\ll\epsilon$. Since $\text{Lip}(\rho)\leq D$, we have $B_{\epsilon_{1}}(\rho^{-1}(a))\cap \rho^{-1}(a-\epsilon)=\emptyset$. Let $0\leq f\leq 1$ be a smooth function on $\rho^{-1}(a)$ and $f=0$ on $\Sigma_{a}:=\rho^{-1}(a)\cap \partial M$. Consider a local smooth vector field $X$ such that $X$ is tangential to $\partial M$ and perpendicular to $\Sigma_{a}$, transverse to $\rho^{-1}(a)$ and inward-pointing. Then we can construct a graph hypersurface $G_{f}:=\{\phi_{\epsilon_{1}f(q)}(q)|q\in \rho^{-1}(a)\}\subset B_{\epsilon_{1}}(\rho^{-1}(a))$, where $\phi_{t}$ is the flow generated by $X$. 
    Next, we compute the interior angle formed by $G_{f}$ and $\partial M$. Let $(x,s)$ be the coordinate of $\rho^{-1}(a)$ near $\Sigma_{a}$, where $x$ is the coordinate in $\Sigma_{a}$ and $\partial s$ is the inner normal direction of $\Sigma_{a}\hookrightarrow \rho^{-1}(a)$. Then we extend the coordinate to $M$ locally by using $\Phi: \Sigma_{a}\times [0,\epsilon_{2})\times [0,\epsilon_{1})\rightarrow M$ satisfying $\Phi(x,s,t)=\phi_{t}((x,s))$. By the construction, $G_{f}$ have the representation $(x,s,\epsilon_{1}f(x,s))$ and the angle between $G_{f}$ and $\partial M$ is acute if  $$\left.\langle \frac{\partial }{\partial s}+\epsilon_{1} \frac{\partial f}{\partial s} \frac{\partial}{\partial t},\frac{\partial}{\partial t}\rangle\right|_{s=0,t=0}>0.$$
    Notice that $\Sigma_{a}$ is compact. Thus, we can choose $f$ such that $\frac{\partial f}{\partial s}|_{s=0}$ is sufficiently large. Then the angle is less than $\frac{\pi}{2}$.
   Notice that $G_{f}\cap \rho^{-1}(a-\epsilon)=\emptyset$ and $\text{dist}(G_{f},\rho^{-1}(a-\epsilon))\geq \frac{\epsilon-\epsilon_{1}}{D}$. Finally, we define $\tilde{\rho}(x)=a$ for $x\in G_{f}$, $\tilde{\rho}(x)=\rho(x)$ for $x\in \{\rho(x)\leq a-\epsilon\}$ and interpolate between $\rho^{-1}(a-\epsilon)$ and $G_{f}$. Moreover, let $\tilde{\rho}(x)\equiv a$ for other $x\in M.$ Thus we obtain a smooth function $\tilde{\rho}$ such that $\text{Lip}(\tilde{\rho})\leq \frac{\epsilon}{\epsilon-\epsilon_{1}}D$. By choosing $\epsilon_1$ small, we have $ \text{Lip}(\tilde{\rho})\leq D+\epsilon'$ for some $\epsilon'>0$.
\end{proof}
\begin{remark}\label{acute angle 2}
    Under the same assumptions of Lemma \ref{acute angle} except that $\partial(\{\rho(x)>-a\})$ is a compact set, there exist a smooth function $\tilde{\rho}$ and $\epsilon'>0$ such that  the angle between $\partial (\{\tilde{\rho}(x)> -a\})\setminus \partial M$ and $\partial (\{\tilde{\rho}(x)> -a\})\cap \partial M$ is less than or equal to $\frac{\pi}{2}$, $\text{Lip}(\tilde{\rho})\leq D+\epsilon'$ and $\tilde{\rho}(x)=\rho(x)$ for $x\in \{\rho(x)\geq -a+\epsilon\}$. 
\end{remark}

Denote $\rho_0=\operatorname{dist}(x,\Sigma_0)$ the signed distance function to $\Sigma_0$ on $\hat{M}$. Let $B_\epsilon(p_k)$ be small balls on $\hat{M}$ along whose boundaries we glue $N_k$ where $p_k\in \tilde{M}_k$. We modify $\rho_0$ by interpolating $\rho_0$ and the constant $A_k$ such that $\rho_0\equiv A_k$ in $B_{\epsilon}(p_k)$ where $A_k=\operatorname{dist}(p_k,\Sigma_0)$. Notice that $A_k> 0$ if $k\geq0$ and $A_k< 0$ if $k<0$. On each $N_k^{'}$ define
\[\rho_0(x)=\begin{cases}
    A_k+\operatorname{dist}(x,\partial N_k^{'}\backslash \partial N_{k})\ k\geq 0;\\
    A_k-\operatorname{dist}(x,\partial N_k^{'}\backslash \partial N_{k}) \ k<0.
\end{cases}\]
Thus, we obtain a function $\rho_0$ defined on $\hat{Y}$ that is Lipschitz.
We define $\rho_1$ to be a smoothing of $\rho_0$, and we can assume that $\rho_1\equiv A_k$ in a small neighborhood of $\partial N_k^{'}\backslash \partial N_{k}$. Notice that there exists a constant $L'>0$ such that
\[|\operatorname{Lip}(\rho_1)|_g\leq \frac{\sqrt{a}}{2}L'.\]
Let $L>L'$ be a constant. Denote
\[\Omega=\{x\in \hat{M}^{'}:|\rho_1|\leq \frac{\pi L}{2}\};\]
\[\Omega_k=\{x\in N_k^{'}: |\rho_1(x)|\leq A_{k}+\frac{2L}{\tan(L^{-1}A_{k})}\}, \ \forall \ k\in \mathbb{N};\]
Denote $I=\max\{k\in\mathbb{N}: |A_k|<\frac{\pi L}{2}\}$. Let $\Pi =\Omega\cup (\cup_{|k|\leq I}\Omega_k)$ be a compact domain of $\hat{Y}$ whose boundary can be split to $\partial\Pi\cap\partial \hat{M}$ and $\partial\Pi\setminus\partial \hat{M}$. By modifying $\rho_1$ slightly we can assume that $\partial\Pi\cap\partial \hat{M}$ and $\partial\Pi\setminus\partial \hat{M}$ intersects transversally. However the dihedral angle formed by them may not be less than or equal to $\pi/2$. We may apply Lemma \ref{acute angle} to obtain a new smooth function $\rho$ on $\hat{Y}$ such that $\partial(\{x\in \hat{M}^{'}:|\rho(x)|<\frac{\pi L}{2}\})\setminus \partial \hat{M}^{'}$ and $\partial(\{x\in N_k^{'}: |\rho(x)|<A_{k}+\frac{2L}{\tan(L^{-1}A_{k})}\})\setminus \partial N_k^{'}$ intersect $\partial \hat{Y}$ with the dihedral angle less than or equal to $\frac{\pi}{2}$. Moreover, $\operatorname{Lip}(\rho)\leq \frac{\sqrt{a}}{2}L $ if we choose $\epsilon<\frac{\sqrt{a}}{2}(L-L')$ in Lemma \ref{acute angle}. Notice that $\Pi\setminus \partial \hat{Y}$ has a finite number of disjoint components. By Remark \ref{acute angle 2}, we can apply  Lemma \ref{acute angle} finitely many times to obtain the desired smooth function $\rho.$


On $\{|\rho|<\frac{\pi L}{2}\}\cap \hat{M}'$, we define 
\[
    h(p) = \frac{1}{\sqrt{a}}\tan(\frac{1}{L}\rho(p)).
\] 
On the rest of $\hat{M}'$, we extend $h$ by taking $ h=\pm \infty$ such that it is continuous. 
If $A_{k}\geq \frac{\pi L}{2}$, set $h=+\infty$ on $ N_{k}'$. If $A_{k}\leq-\frac{\pi L}{2}$, set $h=-\infty$ on $N_{k}'$. For $|A_{k}|< \frac{\pi L}{2}$, 
we define 
\[
h(x) = -\frac{2L}{\sqrt{a}(\rho(x)-A_{k}-\frac{2L}{\tan(L^{-1}A_{k})})}.
\]
on $N_{k}'\cap \{\rho<A_{k}+\frac{2L}{\tan(L^{-1}A_{k})}\}$ as $k\geq 0$ or on $N_{k}'\cap \{\rho>A_{k}+\frac{2L}{\tan(L^{-1}A_{k})}\}$ as $k< 0$.
Similarly, we extend to the rest of $N_{k}'$ by setting $h=\pm \infty$ and smoothen $h$ near the glue part.  In all, we have constructed the smooth function $h$ satisfying \eqref{equ about h}.

Now we can construct a free boundary $\mu$-bubble $\Lambda_{1}$ with respect to $h$. Using proposition \ref{secondvariationof-A}, there exist $(\Lambda_{1},\partial \Lambda_{1})\subset (\hat{Y},\partial\hat{Y})$ which is a free boundary $\mu$-bubble. Next, we want to construct a compact manifold which contains $\Lambda_{1}$. We can find a compact region $Y'\subset \hat{Y}$ with smooth boundary so that $\partial Y' \cap \hat{M} = \Sigma_{J} \cup \Sigma_{-J}$ for some large $J \in \mathbb{N}$. Moreover, $\Lambda_{1}\subset Y'$ and other boundary components of $Y'$ are compact and contained in some $N_k'$. By gluing the hypersurfaces $\Sigma_{J}$ and $\Sigma_{-J}$ to each other, we thus
obtain a manifold $\tilde{Y}$ with boundary diffeomorphic to $\bar{M} \#_{i\in I} \bar{X}_{i}$, where $\bar{M}$ is a $2J$-cyclic covering of $M$
obtained by cutting and pasting along $\Sigma$, $\#$ denotes the interior connected sum and each $\bar{X}_{i}$ is a compact manifold with boundary. This manifold $\tilde{Y}$ contains $\Lambda_1$ as a hypersurface.

On $\Lambda_{1}$, we have $H_{\Lambda_{1}}=h$. For any $\varphi\in C^\infty(\Lambda_{1})$,
    \begin{align*}
    0&\leq \int_{\Lambda_{1}} |\nabla_{\Lambda_{1}}\varphi|^2-(|A^{\Lambda_{1}}|^2+\text{Ric}_{\tilde{Y}}(\nu,\nu)+\nabla_\nu h)\varphi^2-\int_{\partial\Lambda_{1}}A^{\partial \tilde{Y}}(\nu,\nu)\varphi^2\\
    &=\int_{\Lambda_{1}} -\big(\Delta_{\Lambda_{1}}\varphi+(|A^{\Lambda_{1}}|^2+\text{Ric}_{\tilde{Y}}(\nu,\nu)+\nabla_\nu h)\varphi\big)\varphi+\int_{\partial\Lambda_{1}}\varphi(\frac{\partial \varphi}{\partial \eta}-A^{\partial \tilde{Y}}(\nu,\nu)\varphi).
    \end{align*}
    
    \textbf{Case} $m=1$. By plugging $\varphi=1$, we have
    \[
    \begin{aligned}
       \int_{\Lambda_{1}}(|A^{\Lambda_{1}}|^2+\text{Ric}_{\tilde{Y}}(\nu,\nu)+\nabla_\nu h)+\int_{\partial\Lambda_{1}}A^{\partial \tilde{Y}}(\nu,\nu) \leq 0.
    \end{aligned}
    \]
    Taking $a=\frac{1}{n-1}$, so $h$ satisfies satisfies
    \[
    (C_{m})_{\hat{Y}}+\frac{1}{n-1} h^2-2|D_{\hat{Y}}h|>0.
    \]
    Using Cauchy-Schwarz inequality, we have 
    \[
    \begin{aligned}
        |A^{\Lambda_{1}}|^2+\text{Ric}_{\tilde{Y}}(\nu,\nu)+\nabla_\nu h\geq (C_{1})_{\hat{Y}}+\frac{1}{n-1} h^2-|D_{\hat{Y}}h|.
    \end{aligned}
    \]
    By the condition $H_{1}\geq 0$ on $\partial \hat{Y}$ and $\partial \Lambda_{1}\subset \partial \hat{Y}$, we obtain
    \[
    \begin{aligned}
       0<\int_{\Lambda_{1}}(|A^{\Lambda_{1}}|^2+\text{Ric}_{\tilde{Y}}(\nu,\nu)+\nabla_\nu h)+\int_{\partial\Lambda_{1}}A^{\partial \tilde{Y}}(\nu,\nu) \leq 0
    \end{aligned}
    \]
    which is a contradiction.
    
    \textbf{Case} $m\geq 2$. We can consider the eigenfunction $\rho_{1}\in C^{\infty} (\Lambda_{1})$ which satisfies the following equation.
\[
\left\{
\begin{array}{cc}
     -\Delta_{\Lambda_{1}}\rho_{1}-(|A^{\Lambda_{1}}|^2+\text{Ric}(\nu,\nu)+\nabla_{\nu}h)\rho_{1} = \lambda \rho_{1} & \text{in}\ \Lambda_{1}; \\
     \frac{\partial \rho_{1}}{\partial \eta} = \rho_{1}A^{\partial \tilde{Y}}(\nu,\nu)& \text{on}\ \partial \Lambda_{1} .
\end{array}\right.
\]
Now we want to show $\iota:\Lambda_{1}\rightarrow \tilde{Y}$ satisfying the assumptions of Lemma \ref{existence of slicing}. Then we can obtain a $(m-1)$-weighted slicing of $(\Lambda_{1},\rho_{1})$. Consider the map $G=\pi \circ f: \tilde{Y}\rightarrow M$, where $f:\tilde{Y}\rightarrow \bar{M}$ collapse $N_{k}'$ to one point and $\pi:\bar{M}\rightarrow M$ is the covering map.  Then we are able to show $(\iota\circ G)^{*}\alpha_{2}\smile (\iota\circ G)^{*}\alpha_{3}\smile \cdots \smile (\iota\circ G)^{*}\alpha_{m}\neq 0\in H^{m-1}(\Lambda_{1},\mathbb{Z})$. 
In fact,
\[
\begin{aligned}
    & G_{*}([\Lambda_{1}] \frown (G^{*}\alpha_{2}\smile \cdots \smile G^{*}\alpha_{m}))\\
    =& G_{*}[\Lambda_{1}]\frown (\alpha_{2}\smile \cdots \smile \alpha_{m})\\
    =& [\Sigma]\frown (\alpha_{2}\smile \cdots \smile \alpha_{m})\\
    =& ([M]\frown \alpha_1) \frown(\alpha_{2}\smile \cdots \smile \alpha_{m})\\
    =&[M]\frown (\alpha_{1}\smile \cdots \smile \alpha_{m})\\
    \neq & 0.
\end{aligned}
\]
Since $G_{*}: H_{*}(\tilde{Y},\partial \tilde{Y})\rightarrow H_{*}(M,\partial M)$ is a group homomorphism. Then, $[\Lambda_{1}] \frown (G^{*}\alpha_{2}\smile \cdots \smile G^{*}\alpha_{m})\neq 0\in H_{n-m}(\tilde{Y},\partial \tilde{Y})$. Thus, $[\Lambda_{1}] \frown (\iota^{*} G^{*}\alpha_{2}\smile \cdots \smile \iota^{*}G^{*}\alpha_{m})\neq 0\in H_{n-m}(\Lambda_{1},\partial \Lambda_{1})$ and then $(\iota\circ G)^{*}\alpha_{2}\smile (\iota\circ G)^{*}\alpha_{3}\smile \cdots \smile (\iota\circ G)^{*}\alpha_{m}\neq 0\in H^{m-1}(\Lambda_{1},\mathbb{Z})$. Using Lemma \ref{existence of slicing} and Remark \ref{without boundary condition}, there exists a weighted slicing $\Sigma_{m}\subset \Sigma_{m-1}\subset \cdots \subset \Sigma_{2}\subset \Lambda_{1}\subset \tilde{Y}$.

Similar to \eqref{bottom term}, we have
\[
\begin{aligned}
        &\int_{\Sigma_{m}}\rho_{m-1}^{-1}\big(\Delta_{\Sigma_{m-1}}\log \rho_{m-1} +(\nabla_{\Sigma_{m-1}}^{2}\log \rho_{m-1})(\nu_{m},\nu_{m})-|A^{\Sigma_{m}}|^2\\
        &-\text{Ric}_{\Sigma_{m-1}}(\nu_{m},\nu_{m})\big
        )\ d\mu -\int_{\partial \Sigma_{m}}\rho_{m-1}^{-1}(\frac{\partial \log \rho_{m-1}}{\partial \eta}+A^{\partial \Sigma_{m-1}}(\nu_{m},\nu_{m}))\geq 0.
    \end{aligned}
\]
The integration over $\Sigma_{m}$ was estimated in \cite{chenshuli_end}. That is, when $m^2-mn+2n-2>0$ and $m^2-mn+m+n>0$ and $2\leq m \leq n-1$ and $a$ small enough, we can obtain
    \[
    \begin{aligned}
        &\int_{\Sigma_{m}}\rho_{m-1}^{-1}\big(\Delta_{\Sigma_{m-1}}\log \rho_{m-1} +(\nabla_{\Sigma_{m-1}}^{2}\log \rho_{m-1})(\nu_{m},\nu_{m})-|A^{\Sigma_{m}}|^2\\
        &-\text{Ric}_{\Sigma_{m-1}}(\nu_{m},\nu_{m})\big
        )\ d\mu \leq \int_{\Sigma_{m}}-\rho_{m-1}^{-1}\big(C_{m}(\nu_{1},...,\nu_{m})+aH_{\Lambda_{1}}^2-|\langle \nabla_{\tilde{Y}}h,\nu_{1}\rangle |\big)\ d\mu.
    \end{aligned}
    \]
    The integration over $\partial \Sigma_{m}$ is similar to \eqref{eq1}, we have
    \[
    \frac{\partial \log \rho_{m-1}}{\partial \eta}+A^{\partial \Sigma_{m-1}}(\nu_{m},\nu_{m})
        =H_{m}(\nu_{1},...,\nu_{m}).
    \]
    Thus, by inequality \eqref{equ about h} we have
    \begin{equation}\label{inequa bubble}
    \int_{\Sigma_{m}}\rho_{m-1}^{-1}\big(C_{m}(\nu_{1},...,\nu_{m})+aH_{\Lambda_{1}}^2-|\langle \nabla_{\tilde{Y}}h,\nu_{1}\rangle |\big)\ d\mu + \int_{\partial \Sigma_{m}}\rho_{m-1}^{-1}H_{m}(\nu_{1},...,\nu_{m})\leq 0.
    \end{equation}
    This is a contradiction to \eqref{equ about h} and the assumption that $H_m\geq 0$.\\
\vskip.2cm
\subsection{Connected sum in the boundary}

In this subsection, we deal with the case of connected sum in the boundary. Most part of the proof is same as above. We also construct hypersurface $\Sigma$ and an infinite cyclic cover $\hat{M}$ as in the above case. 
Let $Y=M\#_{b} N$. Let $B_p(R)$ be a ball in $\partial M$ with small radius $R$ such that $B_p(R)\cap \Sigma=\emptyset$ and let $B$ be a small ball in $\partial N$. Then we have that $Y=M\cup N/{\sim} $, where $M$ and $N$ are glued along $B_p(R)$ and $B$ by using an oreintation-reversing diffeomorphism. Next, we denote $\hat{Y}=\hat{M}\cup (\cup_{k\in \mathbb{Z}}N_{k})/{\sim}$ as an infinite cyclic covering $\hat{Y}$ of $Y$, where $N_{k}$ are copies of $N$. By our assumption, we can assume that the $m$-intermediate curvature satisfies $(C_{m})_{\hat{Y}}>1$ on $\hat{M}$ and the $m$-intermediate boundary curvature satisfies $H_m>b>0$ on $\partial \hat{M}.$

Denote $\rho_0=\operatorname{dist}(x,\Sigma_0)$ the signed distance function to $\Sigma_0$ on $\hat{M}$. Let $B_\epsilon(p_k)$ be small balls on $\partial \hat{M}$ along it we glue $N_k$ where $p_k\in \tilde{M}_k$. We modify $\rho_0$ by interpolating $\rho_0$ and the constant $A_k$ such that $\rho_0\equiv A_k$ in $B_{\epsilon}(p_k)$ where $A_k=\operatorname{dist}(p_k,\Sigma_0)$. Notice that $A_k> 0$ if $k\geq0$ and $A_k< 0$ if $k<0$. On each $N_k$ define
\[\rho_0(x)=\begin{cases}
    A_k+\operatorname{dist}(x,B_{k})\ k\geq 0;\\
    A_k-\operatorname{dist}(x,B_{k}) \ k<0,
\end{cases}\]
where each $B_{k}$ in $N_{k}$ is a copy of the ball $B$ in $N$. 

Fix $a>0$, using same argument as connected sum in the interior, we can construct a weight function $h$ corresponding to $a$ satisfying 
\begin{equation}\label{equ about h boundary}
(C_{m})_{\hat{Y}}+a h^2-2|D_{\hat{Y}}h|>0.
\end{equation}
and a free boundary $\mu$-bubble $\Lambda_{1}\subset \hat{Y}$ with respect to $h$. Then we find a compact manifold $Y'\subset \hat{Y}$ satisfying $\partial Y' \cap \hat{M} = \Sigma_{J}\cup \Sigma_{-J}$ for some large $J$ and $\Lambda_{1}\subset Y'$. After gluing $\Sigma_{J}$ and $\Sigma_{-J}$, we get a manifold $\tilde{Y}$ with boundary diffeomorphic to $\tilde{M} \#_{i\in I} \tilde{X}_{i}$, where $\tilde{M}$ is a $2J$-cyclic covering of M
obtained by cutting and pasting along $\Sigma$. This manifold $\tilde{Y}$ contains $\Lambda_1$ as a hypersurface. The rest of the proof follows from the same estimates in the case of interior connected sum.

\bibliographystyle{alpha}

\bibliography{reference}
\end{document}